\documentclass[smallextended]{svjour3}       % onecolumn (second format)
\smartqed  % flush right qed marks, e.g. at end of proof
\usepackage{amsfonts}%para algunos tipos de letra, p.e.R(reales)%
\usepackage{amsmath}
\usepackage{graphicx}
\usepackage[english]{babel}

\begin{document}

\title{Bohr's equivalence relation in the space of Besicovitch almost periodic functions}

\author{J.M. Sepulcre \and
        T. Vidal
} \institute{J.M. Sepulcre  \and T. Vidal   \at
              Department of Mathematics\\
              University of Alicante\\
              03080-Alicante, Spain\\
             \email{JM.Sepulcre@ua.es, tmvg@alu.ua.es}
                         %  \\
}

\date{Received: date / Accepted: date}
\maketitle
\begin{abstract}
Based on Bohr's equivalence relation which was established for general Dirichlet series, in this paper we introduce a new equivalence relation on the space of almost periodic functions in the sense of Besicovitch, $B(\mathbb{R},\mathbb{C})$, defined in terms of polynomial approximations. From this, we show that in an important subspace $B^2(\mathbb{R},\mathbb{C})\subset B(\mathbb{R},\mathbb{C})$, where Parseval's equality and Riesz-Fischer theorem holds, its equivalence classes are sequentially compact and  the family of translates of a function belonging to this subspace is dense in its own class.

%This relation is used to show that, with respect to the topology of the important subspace $B^2(\mathbb{R},\mathbb{C})\subset B(\mathbb{R},\mathbb{C})$, every equivalence class, formed by equivalent functions to a given function in $B^2(\mathbb{R},\mathbb{C})$, is sequentially compact and the limit points of the family of translates of a function in $B^2(\mathbb{R},\mathbb{C})$ are precisely the functions which are equivalent to it.

\end{abstract}

\section{Introduction}

The class of almost periodic functions, whose theory was created and developed in its main features by H. Bohr during the $1920$'s, is the class of continuous functions possessing a certain structural property, which is a generalization of pure periodicity. This theory opened a way to study a wide class of trigonometric series of the general type and even exponential series. %We refer the reader to the monograph by Singh and 25 Manhas [16] and the references therein for a comprehensive treatment of this subject.
In this context, we can cite, among others, the papers \cite{Besi,Bochner,Bohr,Bohr2,Corduneanu,Jessen}.

%It is also the class of limit functions of uniformly convergent sequences of trigonometric polynomials.
%
%The theory of almost periodic functions, which was created and developed in its main features by H. Bohr during the $1920$'s, opened a way to study a wide class of trigonometric series of the general type and even exponential series (see for example \cite{Besi,Bohr,Bohr2,Corduneanu1,Corduneanu,Jessen}).
%
%In the case of the functions that are defined on the real numbers, the notion of almost periodicity leads us to generalize purely periodic functions.

Let $f(t)$ be a real or complex function of an unrestricted real variable $t$. The notion of almost periodicity given by Bohr involves the fact that $f(t)$ must be continuous, and for every $\varepsilon>0$ there corresponds a number $l=l(\varepsilon)>0$ such that each interval of length $l$ contains a number $\tau$ satisfying
$|f(t+\tau)-f(t)|\leq\varepsilon$ for all $t$. As in \cite{Corduneanu}, we will denote as $AP(\mathbb{R},\mathbb{C})$ the space of almost periodic functions in the sense of this definition (Bohr's condition). %, which coincides with the notion of uniform almost periodicity used in \cite{Besi}. %As in classical Fourier analysis, every almost periodic function is bounded and is associated with a Fourier series with real frequencies.
A very important result of this
theory is the approximation theorem according to which the class of almost periodic functions $AP(\mathbb{R},\mathbb{C})$ coincides with the class of limit functions of uniformly convergent sequences of trigonometric polynomials of the type
\begin{equation}\label{t1}
a_1e^{i\lambda_1t}+\ldots+a_ne^{i\lambda_nt}
\end{equation}
with arbitrary real exponents $\lambda_j$ and arbitrary complex coefficients $a_j$. Moreover,
%in the context of the space of the bounded and continuous functions on $\mathbb{R}$,
S. Bochner observed that Bohr's notion of almost periodicity of a function $f$ is equivalent to the relative compactness, in the sense of uniform convergence, of the family of its translates $\{f(t+h)\}$, $h\in\mathbb{R}$.

In
the course of time, some variants and extensions of Bohr's concept have been
introduced, most notably by A. S. Besicovitch, W. Stepanov and H. Weyl. We refer the reader to the papers by Besicovitch \cite[Chapter II]{Besi}, Bohr and F{\oe}lner \cite{BohrFolner}, Corduneanu \cite{Corduneanu}, and by Andres, Bersani and Grande \cite{Andres} and the references therein for a comprehensive treatment of this subject.

In particular, A.S. Besicovitch enlarged the class of almost periodic functions by considering the convergence of sequences of functions in a more general sense than uniform convergence. In this way, the Besicovitch spaces $B^p(\mathbb{R},\mathbb{C})$, $1\leq p<\infty$,
are obtained by the completion of the trigonometric polynomials of the form (\ref{t1}) with respect to the seminorms
$$\left(\limsup_{l\to\infty}(2l)^{-1}\int_{-l}^l|f(t)|^p dt\right)^{1/p}.$$ This topology is certainly weaker than that of the uniform convergence.
In particular, the space $B^1(\mathbb{R},\mathbb{C})$ is denoted by $B(\mathbb{R},\mathbb{C})$ and contains $AP(\mathbb{R},\mathbb{C})$, $B^2(\mathbb{R},\mathbb{C})$ and all variants of almost periodic functions which were mentioned above.
%it is satisfied that $AP(\mathbb{R},\mathbb{C})\subset S^2(\mathbb{R},\mathbb{C})\subset W^2(\mathbb{R},\mathbb{C})\subset B^2(\mathbb{R},\mathbb{C})\subset B(\mathbb{R},\mathbb{C})$. %and $AP(\mathbb{R},\mathbb{C})\subset S(\mathbb{R},\mathbb{C})\subset B(\mathbb{R},\mathbb{C})$. %The space $B^2(\mathbb{R},\mathbb{C})$ is also denoted as $AP_2(\mathbb{R},\mathbb{C})$.
%The space of almost periodic functions in Besicovitch sense is the closure in
% $M_2(\mathbb{R},\mathbb{C})$ of the linear manifold $T+L$, with
%$T$ standing for the set of trigonometric polynomials of the form (\ref{t1}) and $L$ defined above. Again, it is worth noting that $AP_2(\mathbb{R},\mathbb{C})$ is a Banach space over $\mathbb{C}$ with the norm induced by $\|\cdot\|_M$, namely
%$$\|f+L\|_M=\|f\|_M,\ f\in AP_2(\mathbb{R},\mathbb{C}).$$
%We know about the following inclusions of the spaces of almost
%periodic functions $$AP_1(\mathbb{R},\mathbb{C})\subset AP(\mathbb{R},\mathbb{C})\subset S^2(\mathbb{R},\mathbb{C})\subset AP_2(\mathbb{R},\mathbb{C})=B^2(\mathbb{R},\mathbb{C})\subset B(\mathbb{R},\mathbb{C}).$$ %Unfortunately, we cannot interpose S(R,C) between AP(R,C) and AP 2 (R,C).
%This is because the functions in S(R,C) are only locally integrable, while
%those in AP 2 (R,C) are locally square integrable – a stronger property than
%local integrability
%In fact, $B(\mathbb{R},\mathbb{C})$ contains all spaces of almost periodic functions
%we have discussed so far.

Moreover, for any function $f\in B(\mathbb{R},\mathbb{C})$ there exists the mean value
\begin{equation}\label{Meanvalo}
M(f)=\lim_{l\to\infty}(2l)^{-1}\int_{-l}^lf(t)dt
\end{equation}
and, at most, a countable
set of values of $\lambda_k\in\mathbb{R}$  such that
$a_k=a(f,\lambda_k) = M(f(t)e^{-\lambda_k t}) \neq 0$.
Thus the series
$\sum_{k\geq 1}a_ke^{i\lambda_kt}$
is called the Fourier series of $f$ \cite[Section 4.2]{Corduneanu}. %--revisar numeración
Also, $\lambda_k$ and $a_k$ are called the Fourier
exponents and coefficients of the
function $f$, respectively. %This fact is represented by $$f\sim \sum_{k\geq 1}a_ke^{i\lambda_kt}.$$

In the case that $f\in B^2(\mathbb{R},\mathbb{C})$, with $\sum_{k\geq 1}a_ke^{i\lambda_kt}$ the Fourier series of $f$, it is accomplished the Parseval's equality \cite[p. 109]{Besi}
$$\sum_{k\geq 1}|a_k|^2=M(|f(t)|^2)<\infty.$$ In this respect, if $f_1,f_2\in B^2(\mathbb{R},\mathbb{C})$, then $f_1$ and $f_2$ have the same Fourier series if and only if
$M(|f_1(t)-f_2(t)|^2)=0.$ That is, two functions satisfying this condition belong to the same class of equivalence defined in terms of the Fourier series. This equivalence relation is inherent to the classes $B^p(\mathbb{R},\mathbb{C})$ and it is different
from the generalization of Bohr's equivalence of Definition \ref{DefEquiv00} which is the main
tool of this paper.

Besicovitch's generalization is interesting because, for this
extension, the analogue of the Riesz-Fischer theorem is also valid,
that is to say, any trigonometric series $\sum_{n\geq 1} a_ne^{i\lambda_nt}$ with $\sum_{n\geq 1}|a_n|^2$ finite is the
Fourier series of a $B^2(\mathbb{R},\mathbb{C})$ almost periodic function \cite[p. 110]{Besi} (in this sense, $B^2(\mathbb{R},\mathbb{C})$ is also called $AP_2(\mathbb{R},\mathbb{C})$ in \cite{Corduneanu}). This is not the case for some Stepanov or Weyl functions \cite{Jessen}. As a consequence of the above, a Fourier series $\sum_{n\geq 1}a_ne^{i\lambda_nt}$ so that $\sum_{n\geq 1}|a_n|^2<\infty$ represents an equivalence class of functions in $B^2(\mathbb{R},\mathbb{C})$. %(not a single function). %--corrobora con Tomás la frase añadida

In this paper we extend Bohr's equi\-valence relation to the Fourier series associated with the Besicovitch almost periodic functions, and hence to the Besicovitch almost periodic functions too. In this way, in view of the analogue of the Riesz-Fischer theorem and with respect to the topology of $B^2(\mathbb{R},\mathbb{C})$, the main result of our paper shows that, fixed an almost periodic function in $B^2(\mathbb{R},\mathbb{C})$, the limit points of the set of its translates are precisely the functions which are equiva\-lent to it (see Theorem \ref{mth0} in this paper).
This means that the Bochner-type property, which is satisfied for the Besicovitch classes of almost periodic functions defined as above in terms of polynomials approximations (see \cite[Definition 5.5, Definition 5.17 and Theorem 5.34]{Andres} or \cite[Section 3.4, p. 65]{Corduneanu}), is now refined for $B^2(\mathbb{R},\mathbb{C})$ in the sense that we show that the condition of almost periodicity in the Besicovitch sense implies that every sequence of translates has a subsequence that converges in $B^2(\mathbb{R},\mathbb{C})$ to an equivalent function. %--comprobar numeración de Corduneanu HECHO

%In particular, in terms of Voronin's universality theorem,
%we show that any exponential sum which is equivalent to the Riemann zeta function can
%be uniformly approximated in $\{s=\sigma+it:\sigma>1\}$ by certain vertical translates of the Riemann zeta-function (see Theorem \ref{mtrzf} in this paper). For example, we assure the
%existence of two increasing unbounded sequences $\{\tau_n\}_{n\geq 1}$ and $\{\varsigma_n\}_{n\geq 1}$
%of positive numbers such that the sequences of functions
%$\{\zeta(s+i\tau_n)\}$ and $\{\zeta(s+i\varsigma_n)\}$, $n\in\mathbb{N}$, converge uniformly to
%$\zeta(s)$ and $\zeta_{\lambda}(s)$ on every reduced strip of
%$\{s\in\mathbb{C}:\operatorname{Re}s>1\}$ respectively, where $\zeta_{\lambda}(s):=\sum_{n\geq 1}\frac{\lambda(n)}{n^s}$ is the ordinary Dirichlet series for the Liouville function $\lambda(n)$.
%To the best of our knowledge, these results
%have not been considered in the literature.
%Finally, as a
%consequence, we will obtain an alternative
%demonstration of a known result related to the infimum of
%$|\zeta(s)|$ on certain regions in the half-plane $\sigma\geq 1$
%(see corollaries \ref{positive} and \ref{ult} in this paper).

\section{Preliminary definitions and results on exponential sums}

We shall refer to the expressions of the type
$$P_1(p)e^{\lambda_1p}+\ldots+P_j(p)e^{\lambda_jp}+\ldots$$
as exponential sums, where the frequencies $\lambda_j$ are complex numbers and the $P_j(p)$ are polynomials in $p$.  In this paper we are going to consider some functions which are associated with a concrete
subclass of these exponential sums, where the parameter $p$ will be changed by $t$ in the real case. In this way, as in \cite{SV}, we take the following definition.

\begin{definition}
Let $\Lambda=\{\lambda_1,\lambda_2,\ldots,\lambda_j,\ldots\}$ be an arbitrary countable set of distinct real numbers, which we will call a set of exponents or frequencies. We will say that an exponential sum is in the class $\mathcal{S}_{\Lambda}$ if it is a formal series of type
 \begin{equation}\label{eqqnew}
\sum_{j\geq 1}a_je^{\lambda_jp},\ a_j\in\mathbb{C},\ \lambda_j\in\Lambda.
\end{equation}
%Also, we will say that $a_1,a_2,\ldots,a_j,\ldots$ are the coefficients of this exponential sum.
\end{definition}
%
%It is clear that expression (\ref{eqqnew}) is not necessarily associated with a convergent series that defines an holomorphic function. Precisely, we call it formal series to distinguish it from an ordinary
%series. In fact, in the theory of formal series, the parameter $p$ of (\ref{eqqnew}) is never
%assigned a numerical value and questions of convergence or divergence are
%not of interest. In this respect, we operate on formal series algebraically as though they were
%convergent series and the expression $e^{\lambda_jp}$ is simply a device for
%locating the position of the $j$th coefficient $a_j$. In this way, if $A_1(p)=\sum_{j\geq 1}a_je^{\lambda_jp}$ and $A_2(p)=\sum_{j\geq 1}b_je^{\lambda_jp}$ are two formal series in $\mathcal{S}_{\Lambda}$, then $A_1(p)=A_2(p)$ means that $a_j=b_j$ for each $j\geq 1$. Moreover, $A_1(p)+A_2(p):=\sum_{j\geq 1}(a_j+b_j)e^{\lambda_jp}$ and $A_1(p+h):=\sum_{j\geq 1}a_je^{\lambda_jh}e^{\lambda_jp}$ for any $h\in\mathbb{C}$.
%

We next introduce an equivalence relation on the classes $\mathcal{S}_{\Lambda}$.

\begin{definition}\label{DefEquiv00}%(Generalized Bohr's equivalence relation)
Given an arbitrary countable set $\Lambda=\{\lambda_1,\lambda_2,\ldots,\lambda_j,\ldots\}$ of distinct real numbers, consider $A_1(p)$ and $A_2(p)$ two exponential sums in the class $\mathcal{S}_{\Lambda}$, say
$A_1(p)=\sum_{j\geq1}a_je^{\lambda_jp}$ and $A_2(p)=\sum_{j\geq1}b_je^{\lambda_jp}.$
We will say that $A_1$ is equivalent to $A_2$ (in that case, we will write $A_1\shortstack{$_{{\fontsize{6}{7}\selectfont *}}$\\$\sim$} A_2$) if for each integer value $n\geq 1$, with $n\leq \sharp\Lambda$,
there exists a $\mathbb{Q}$-linear map $\psi_n:V_n\to\mathbb{R}$, where $V_n$ is the $\mathbb{Q}$-vector space generated by $\{\lambda_1,\lambda_2,\ldots,\lambda_n\}$, such that
$$b_j=a_je^{i\psi_n(\lambda_j)},\ j=1,\ldots,n.$$
\end{definition}

It is plain that the relation $\shortstack{$_{{\fontsize{6}{7}\selectfont *}}$\\$\sim$}$ considered in the foregoing definition is an equiva\-lence relation. %The reader can also observe that this equivalence relation is a modification of that of \cite[p.173]{Apostol} which was defined for general Dirichlet series. Bohr used it in that case in order to get so-called Bohr's equivalence theorem.

Let
$G_{\Lambda}=\{g_1, g_2,\ldots, g_k,\ldots\}$ be a basis of the
vector space over the rationals generated by a set $\Lambda$ of exponents, %group $W=\mathbb{Z}w_1+\mathbb{Z}w_2+\ldots+\mathbb{Z}w_j+\ldots$,
which implies that $G_{\Lambda}$ is linearly independent over the rationals and each $\lambda_j$ is expressible as a finite linear combination of terms of $G_{\Lambda}$, say
\begin{equation}\label{errej}
\lambda_j=\sum_{k=1}^{q_j}r_{j,k}g_k,\ \mbox{for some }r_{j,k}\in\mathbb{Q}.
\end{equation}
By abuse of notation, we will say that $G_{\Lambda}$ is a basis for $\Lambda$. Moreover, we will say that $G_{\Lambda}$ is an integral basis for $\Lambda$ when $r_{j,k}\in\mathbb{Z}$ for any $j,k$. By taking this
into account, the equivalence relation introduced in Definition \ref{DefEquiv00} can be characterized in terms of a basis for $\Lambda$ (see \cite[Proposition 1']{SV}). %--dar cita de la siguiente proposición HECHO
\begin{proposition}\label{DefEquiv}
Given $\Lambda=\{\lambda_1,\lambda_2,\ldots,\lambda_j,\ldots\}$ a set of exponents, consider $A_1(p)$ and $A_2(p)$ two exponential sums in the class $\mathcal{S}_{\Lambda}$, say
$A_1(p)=\sum_{j\geq1}a_je^{\lambda_jp}$ and $A_2(p)=\sum_{j\geq1}b_je^{\lambda_jp}.$
Fixed a basis $G_{\Lambda}$ for $\Lambda$, for each $j\geq1$ let $\mathbf{r}_j$ be the vector of rational components satisfying (\ref{errej}).  %-%$$\lambda_j=<\mathbf{r}_j,\mathbf{g}>=\sum_{k=1}^{q_j}r_{j,k}g_k,$$ where
  %$\mathbf{g}=(g_1,g_2,\ldots,g_k,\ldots)$ is the vector of the elements of a basis $G_{\Lambda}$ for $\Lambda$.
Then $A_1\shortstack{$_{{\fontsize{6}{7}\selectfont *}}$\\$\sim$} A_2$ %, relative to the basis $G_{\Lambda}$,
if and only if for each integer value $n\geq1$, with $n\leq\sharp\Lambda$, there exists a vector $\mathbf{x}_n=(x_{n,1},x_{n,2},\ldots,x_{n,k},\ldots)\in \mathbb{R}^{\sharp G_{\Lambda}}$
such that $b_j=a_j e^{<\mathbf{r}_j,\mathbf{x}_n>i}$ for $j=1,2,\ldots,n$.

\noindent Furthermore, if $G_{\Lambda}$ is an integral basis for $\Lambda$ then $A_1\shortstack{$_{{\fontsize{6}{7}\selectfont *}}$\\$\sim$} A_2$ %, relative to the basis $G_{\Lambda}$,
if and only if there exists $\mathbf{x}_0=(x_{0,1},x_{0,2},\ldots,x_{0,k},\ldots)\in \mathbb{R}^{\sharp G_{\Lambda}}$
such that $b_j=a_j e^{<\mathbf{r}_j,\mathbf{x}_0>i}$ for every $j\geq 1$. %In that case, we will write $A_1\sim A_2$.
\end{proposition}
\begin{proof}
For each integer value $n\geq 1$, let $V_n$ be the $\mathbb{Q}$-vector space generated by $\{\lambda_1,\ldots,\lambda_n\}$, $V$ the $\mathbb{Q}$-vector space generated by $\Lambda$, and
$G_{\Lambda}=\{g_1, g_2,\ldots, g_k,\ldots\}$ a basis of $V$.
If $A_1\shortstack{$_{{\fontsize{6}{7}\selectfont *}}$\\$\sim$} A_2$, by Definition \ref{DefEquiv00} for each integer value $n\geq 1$, with $n\leq\sharp\Lambda$, there exists a $\mathbb{Q}$-linear map $\psi_n:V_n\to\mathbb{R}$ such that
$b_j=a_je^{i\psi_n(\lambda_j)},\ j=1,2\ldots,n.$
Hence
$b_j=a_je^{i\sum_{k=1}^{i_j}r_{j,k}\psi_n(g_k)},\ j=1,2\ldots,n$
or, equivalently,
$b_j=a_j e^{i<\mathbf{r}_j,\mathbf{x}_n>},\ j=1,2\ldots,n,$
with $\mathbf{x}_n:=(\psi_n(g_1),\psi_n(g_2),\ldots)$. Conversely, suppose the existence, for each integer value $n\geq 1$, of a vector of the form $\mathbf{x}_n=(x_{n,1},x_{n,2},\ldots,x_{n,k},\ldots)\in \mathbb{R}^{\sharp G_{\Lambda}}$
such that $b_j=a_j e^{<\mathbf{r}_j,\mathbf{x}_n>i}$, $j=1,2\ldots,n$. Thus a $\mathbb{Q}$-linear map $\psi_n:V_n\to\mathbb{R}$ can be defined from $\psi_n(g_k):=x_{n,k}$, $k\geq 1$. Therefore $\psi_n(\lambda_j)=\sum_{k=1}^{i_j}r_{j,k}\psi(g_k)=$ $<\mathbf{r}_j,\mathbf{x}_n>,\ j=1,2\ldots,n,$
and the result follows.

Now, suppose that $G_{\Lambda}$ is an integral basis for $\Lambda$ and $A_1\shortstack{$_{{\fontsize{6}{7}\selectfont *}}$\\$\sim$} A_2$.
Thus, by above, for each fixed integer value $n\geq 1$, let $\mathbf{x}_n=(x_{n,1},x_{n,2},\ldots)\in\mathbb{R}^{\sharp G_{\Lambda}}$ be a vector such that
$b_j=a_j e^{i<\mathbf{r}_j,\mathbf{x}_n>},\ j=1,2\ldots,n.$ Since each component of $\mathbf{r}_j$ is an integer number, without loss of generality, we can take $\mathbf{x}_n\in[0,2\pi)^{\sharp G_{\Lambda}}$ as the unique vector in $[0,2\pi)^{\sharp G_{\Lambda}}$ satisfying the above equalities, where we assume $x_{n,k}=0$ for any $k$ such that $r_{j,k}=0$ for $j=1,\ldots,n$. Therefore, under this assumption, if $m>n$ then $x_{m,k}=x_{n,k}$ for any $k$ so that $x_{n,k}\neq 0$. In this way, we can construct a vector $\mathbf{x}_0=(x_{0,1},x_{0,2},\ldots,x_{0,k},\ldots)\in [0,2\pi)^{\sharp G_{\Lambda}}$ such that $b_j=a_j e^{<\mathbf{r}_j,\mathbf{x}_0>i}$ for every $j\geq 1$. Indeed, if $r_{1,k}\neq 0$ then the component $x_{0,k}$ is chosen as $x_{1,k}$, and if $r_{1,k}=0$ then
each component $x_{0,k}$ is defined as $x_{n+1,k}$ where $r_{j,k}=0$ for $j=1,\ldots,n$ and $r_{n+1,k}\neq 0$.
%
%Now, by reductio ad absurdum, suppose that there does not exist a vector $\mathbf{x}_0=(x_{0,1},x_{0,2},\ldots,x_{0,k},\ldots)\in [0,2\pi)^{\sharp G_{\Lambda}}$
%such that $b_j=a_j e^{<\mathbf{r}_j,\mathbf{x}_0>i}$ for every $j\geq 1$. This implies that there exist $n,k\in \mathbb{N}$ such that $b_n=a_n e^{<\mathbf{r}_n,\mathbf{x}_n>i}$, $b_{n+1}=a_{n+1} e^{<\mathbf{r}_{n+1},\mathbf{x}_{n+1}>i}$ and $x_{n,k}\neq x_{n+1,k}$.
Conversely, if there exists $\mathbf{x}_0=(x_{0,1},x_{0,2},\ldots,x_{0,k},\ldots)\in \mathbb{R}^{\sharp G_{\Lambda}}$
such that $b_j=a_j e^{<\mathbf{r}_j,\mathbf{x}_0>i}$ for every $j\geq 1$, then it is clear that $A_1\shortstack{$_{{\fontsize{6}{7}\selectfont *}}$\\$\sim$} A_2$ under Definition \ref{DefEquiv00}.
\end{proof}

On the other hand, we will say that $G_{\Lambda}$ is the \textit{natural basis} for $\Lambda$, and we will denote it as $G_{\Lambda}^*$, when it is constituted by elements in $\Lambda$. That is, firstly if $\lambda_1\neq 0$ then $g_1:=\lambda_1\in G_{\Lambda}^*$. Secondly, if $\{\lambda_1,\lambda_2\}$ are $\mathbb{Q}$-rationally independent, then $g_2:=\lambda_2\in G_{\Lambda}^*$. Otherwise, if $\{\lambda_1,\lambda_3\}$ are $\mathbb{Q}$-rationally independent, then $g_2:=\lambda_3\in G_{\Lambda}^*$, and so on. In this way, %$r_{1,1}=1$, $r_{1,k}=0$ for $k\neq 1$, $r_{2,2}=1$, $r_{2,k}=0$ for $k\neq 2$. In general,
if $\lambda_j\in G_{\Lambda}^*$ then $r_{j,m_j}=1$ and $r_{j,k}=0$ for $k\neq m_j$, where $m_j$ is such that $g_{m_j}=\lambda_j$. In fact, each element in $G_{\Lambda}^*$ is of the form $g_{m_j}$ for $j$ such that $\lambda_j$ is $\mathbb{Q}$-linear independent of the previous elements in the basis. %i.e. $\sharp\{m_1,m_2,\ldots,m_j,\ldots: \lambda_j\in G_{\Lambda}\}=\sharp G_{\Lambda}$.
Furthermore, if $\lambda_j\notin G_{\Lambda}^*$ then $\lambda_j=\sum_{k=1}^{i_j}r_{j,k}g_k$, where $\{g_{1},g_{2},\ldots,g_{i_j}\}\subset \{\lambda_1,\lambda_2,\ldots,\lambda_{j-1}\}$.

In terms of the natural basis, we next prove another characterization which will be used later.

\begin{proposition}
Given $\Lambda=\{\lambda_1,\lambda_2,\ldots,\lambda_j,\ldots\}$ a set of exponents, consider $A_1(p)$ and $A_2(p)$ two exponential sums in the class $\mathcal{S}_{\Lambda}$, say
$A_1(p)=\sum_{j\geq1}a_je^{\lambda_jp}$ and $A_2(p)=\sum_{j\geq1}b_je^{\lambda_jp}.$
Fixed the natural basis $G_{\Lambda}^*=\{g_1,g_2,\ldots,g_k,\ldots\}$ for $\Lambda$, for each $j\geq1$ let $\mathbf{r}_j\in \mathbb{R}^{\sharp G_{\Lambda}^*}$ be the vector of rational components verifying (\ref{errej}).  %-%$$\lambda_j=<\mathbf{r}_j,\mathbf{g}>=\sum_{k=1}^{q_j}r_{j,k}g_k,$$ where
  %$\mathbf{g}=(g_1,g_2,\ldots,g_k,\ldots)$ is the vector of the elements of a basis $G_{\Lambda}$ for $\Lambda$.
Then $A_1\shortstack{$_{{\fontsize{6}{7}\selectfont *}}$\\$\sim$} A_2$ %, relative to the basis $G_{\Lambda}$,
if and only if there exists $\mathbf{x}_0=(x_{0,1},x_{0,2},\ldots,x_{0,k},\ldots)\in[0,2\pi)^{\sharp G_{\Lambda}^*}$
such that for each $j=1,2,\ldots$ it is satisfied $b_j=a_j e^{<\mathbf{r}_j,\mathbf{x}_0+\mathbf{p}_j>i}$ for some $\mathbf{p}_j=(2\pi n_{j,1},2\pi n_{j,2},\ldots)\in \mathbb{R}^{\sharp G_{\Lambda}^*}$, with $n_{j,k}\in\mathbb{Z}$.
\end{proposition}
\begin{proof}
Suppose that $A_1\shortstack{$_{{\fontsize{6}{7}\selectfont *}}$\\$\sim$} A_2$. Consider $I=\{1,2,\ldots,k,\ldots: \lambda_k\in G_{\Lambda}^*\}$ and $I_n=\{1,2,\ldots,k,\ldots,n: \lambda_k\in G_{\Lambda}^*\}$.
Let $j\in I$, then $r_{j,m_j}=1$ and $r_{j,k}=0$ for $k\neq m_j$, where $m_j$ is such that $g_{m_j}=\lambda_j$.
Thus, by Proposition \ref{DefEquiv}, let $\mathbf{x}_j=(x_{j,1},x_{j,2},\ldots)\in\mathbb{R}^{\sharp G_{\Lambda}^*}$ be a vector such that
$$b_j=a_j e^{i<\mathbf{r}_j,\mathbf{x}_j>}=a_j e^{i\sum_{k=1}^{i_j}r_{j,k}x_{j,k}}=a_je^{ir_{j,m_j}x_{j,m_j}}=a_je^{ix_{j,m_j}}.$$
Define $\mathbf{x}_0=(x_{0,1},x_{0,2},\ldots)\in\mathbb{R}^{\sharp G_{\Lambda}^*}=\mathbb{R}^{\sharp I}$ as $x_{0,m_j}:=x_{j,m_j}$ for $j\in I$. Thus, by taking $\mathbf{p}_j=(0,0,\ldots)$, the result trivially holds for those $j$'s such that $\lambda_j\in G_{\Lambda}^*$, i.e. for $j\in I$. Now, let $j$ be such that $\lambda_j\notin G_{\Lambda}^*$, i.e. $j\notin I$. By Proposition \ref{DefEquiv}, let $\mathbf{x}_j=(x_{j,1},x_{j,2},\ldots)\in\mathbb{R}^{\sharp G_{\Lambda}^*}$ be a vector such that
$$b_p=a_p e^{i<\mathbf{r}_p,\mathbf{x}_j>}=a_p e^{i\sum_{k=1}^{i_j}r_{p,k}x_{j,k}},\ p=1,2,\ldots,j.$$
Note that if $p=1,2,\ldots,j$ is such that $\lambda_p\in G_{\Lambda}^*$, then
$$b_p=a_pe^{ir_{p,m_p}x_{j,m_p}},$$
which necessarily implies that $r_{p,m_p}x_{j,m_p}=r_{p,m_p}x_{p,m_p}+2\pi n_p$, i.e. $x_{j,m_p}=x_{p,m_p}+2\pi n_{j,p}$ for some $n_{j,p}\in\mathbb{Z}$. %Moreover, note that $i_j=\sharp I_{j-1}$. %$\{r_{j,1},r_{j,2},\ldots,r_{j,i_j}\}=\{r_{j,m_1},r_{j,m_2},\ldots,r_{j,m_{j-1}}:m_k\mbox{ is so that }\lambda_{m_k}\in G_{\Lambda}\ \mbox{for }k=1,\ldots,j-1\}$.
Hence
$$b_j=a_j e^{i<\mathbf{r}_j,\mathbf{x}_j>}=a_j e^{i\sum_{k=1}^{i_j}r_{j,k}x_{j,k}}=a_j e^{i\sum_{p\in I_{j-1}}r_{j,m_p}x_{j,m_p}}=$$
$$a_j e^{i\sum_{p\in I_{j-1}}r_{j,m_p}(x_{p,m_p}+2\pi n_{j,p})}=a_j e^{i<\mathbf{r}_j,\mathbf{x}_0+\mathbf{p}_j>},$$
where $\mathbf{p}_j=(2\pi n_{j,1},2\pi n_{j,2},\ldots,0,0,\ldots)$. Moreover, by changing conveniently the vectors $\mathbf{p}_j$, we can take $\mathbf{x}_0\in [0,2\pi)^{\sharp G_{\Lambda}^*}$ without loss of generality.

Conversely, suppose the existence of $\mathbf{x}_0=(x_{0,1},x_{0,2},\ldots,x_{0,k},\ldots)\in\mathbb{R}^{\sharp G_{\Lambda}^*}$ satisfying $b_j=a_j e^{<\mathbf{r}_j,\mathbf{x}_0+\mathbf{p}_j>i}$ for some $\mathbf{p}_j=(2\pi n_{j,1},2\pi n_{j,2},\ldots)\in \mathbb{R}^{\sharp G_{\Lambda}^*}$, with $n_{j,k}\in\mathbb{Z}$. Let $r_{j,k}=\frac{p_{j,k}}{q_{j,k}}$ with $p_{j,k}$ and $q_{j,k}$ coprime integer numbers, and define $q_{n,k}:=\operatorname{lcm}(q_{1,k},q_{2,k},\ldots,q_{n,k})$ for each $k=1,2,\ldots$.
Thus, for any integer number $n\geq 1$, take $\mathbf{x}_n=\mathbf{x}_0+\mathbf{m}_n$, where  $m_{n,k}=2\pi p_{1,k}p_{2,k}\cdots p_{n,k}q_{n,k}$, $k=1,2,\ldots$.
Therefore,
it is satisfied $b_j=a_j e^{<\mathbf{r}_j,\mathbf{x}_n>i}$ for each $j=1,2,\ldots,n$, which implies that $A_1\shortstack{$_{{\fontsize{6}{7}\selectfont *}}$\\$\sim$} A_2$.
\end{proof}

As corollary, we can formulate the following result.
\begin{corollary}\label{use}
Given $\Lambda=\{\lambda_1,\lambda_2,\ldots,\lambda_j,\ldots\}$ a set of exponents, consider $A_1(p)$ and $A_2(p)$ two exponential sums in the class $\mathcal{S}_{\Lambda}$, say
$A_1(p)=\sum_{j\geq1}a_je^{\lambda_jp}$ and $A_2(p)=\sum_{j\geq1}b_je^{\lambda_jp}.$
Fixed a basis $G_{\Lambda}=\{g_1,g_2,\ldots,g_k,\ldots\}$ for $\Lambda$, for each $j\geq1$ let $\mathbf{r}_j\in \mathbb{R}^{\sharp G_{\Lambda}}$ be the vector of rational components verifying (\ref{errej}).  %-%$$\lambda_j=<\mathbf{r}_j,\mathbf{g}>=\sum_{k=1}^{q_j}r_{j,k}g_k,$$ where
  %$\mathbf{g}=(g_1,g_2,\ldots,g_k,\ldots)$ is the vector of the elements of a basis $G_{\Lambda}$ for $\Lambda$.
Then $A_1\shortstack{$_{{\fontsize{6}{7}\selectfont *}}$\\$\sim$} A_2$ %, relative to the basis $G_{\Lambda}$,
if and only if there exists $\mathbf{x}_0=(x_{0,1},x_{0,2},\ldots,x_{0,k},\ldots)\in[0,2\pi)^{\sharp G_{\Lambda}}$
such that for each $j=1,2,\ldots$ it is satisfied $b_j=a_j e^{<\mathbf{r}_j,\mathbf{x}_0+\mathbf{q}_j>i}$ for some $\mathbf{q}_j\in \mathbb{R}^{\sharp G_{\Lambda}}$ which are of the form
$\mathbf{q}_j=T\cdot \mathbf{p}_j^t$, where $T$ is the change of basis matrix, with respect to the natural basis, and $\mathbf{p}_j$ is of the form $(2\pi n_{j,1},2\pi n_{j,2},\ldots,2\pi n_{j,k},\ldots)$, $n_{j,k}\in\mathbb{Z}$.
\end{corollary}

In particular, note that the coefficients of equivalent exponential sums have the same modulus.

\section{The finite exponential sums of the classes $\mathcal{P}_{\mathbb{R},\Lambda}$}\label{section3}

This section is focused on the following classes of finite exponential sums.

\begin{definition}
Let $\Lambda=\{\lambda_1,\ldots,\lambda_n\}$ be a set of $n\geq 1$ distinct real numbers. We will say that a function  $f:\mathbb{R}\mapsto\mathbb{C}$ is in the class $\mathcal{P}_{\mathbb{R},\Lambda}$ if it is of the form
\begin{equation}\label{eqq0new2}
f(t)=a_1e^{i\lambda_1t}+\ldots+a_ne^{i\lambda_nt},\ a_j\in\mathbb{C},\ \lambda_j\in\Lambda,\ j=1,\ldots,n.
\end{equation}
\end{definition}

The functions $f(t)$ of type (\ref{eqq0new2}) are also called trigonometric polynomials.

Note that Definition \ref{DefEquiv00} can be particularized to the classes $\mathcal{P}_{\mathbb{R},\Lambda}$. Furthermore, if $\Lambda$ is finite it is clear that it is always possible to find an integral basis for $\Lambda$.
%The equivalence relation which was introduced in Definition \ref{DefEquiv} can naturally be applied to the exponential sums in $\mathcal{P}_{\mathbb{R},\Lambda}$.
In this context, we next prove the following important result.

\begin{theorem}\label{prop3}
Given $\Lambda=\{\lambda_1,\lambda_2,\ldots,\lambda_n\}$ a set of $n\geq 1$ exponents, let
$a_1e^{i\lambda_1t}+\ldots+a_ne^{i\lambda_nt} \mbox{ and }b_1e^{i\lambda_1t}+\ldots+b_ne^{i\lambda_nt}$ be
two equivalent functions in the class $\mathcal{P}_{\mathbb{R},\Lambda}$. Fixed $d>0$ and $\varepsilon>0$, there exists $\tau>d$ such that
$$\sum_{j=1}^n|a_je^{i\lambda_j\tau}-b_j|<\varepsilon.$$
%$$|f_1(t+\tau)-f_2(t)|<\varepsilon\ \ \forall t\in\mathbb{R}.$$
\end{theorem}
\begin{proof}
Let $G_{\Lambda}=\{g_1,\ldots, g_m\}$, for a certain $m\geq1$, be linearly independent over the rationals so that each $\lambda_j\in\Lambda$ is expressible as a linear combination of its terms, say
\begin{equation}\label{u}
\lambda_j=\sum_{k=1}^{m}r_{j,k}g_k,\ \mbox{for some }r_{j,k}=\frac{p_{j,k}}{q_{j,k}}\in\mathbb{Q},\ j=1,2,\ldots,n.
\end{equation}
Consider $\varepsilon>0$, $q:=\operatorname{lcm}(q_{j,k}: j=1,\ldots,n,k=1,\ldots,m)$, $r:=\max\{|r_{j,k}|: j=1,\ldots,n,k=1,\ldots,m\}>0$ and $a:=\max\{|a_j|:j=1,2,\ldots,n\}>0$.
 Since $a_1e^{i\lambda_1t}+\ldots+a_ne^{i\lambda_nt}$ and $b_1e^{i\lambda_1t}+\ldots+b_ne^{i\lambda_nt}$ are equivalent, Proposition \ref{DefEquiv} assures the existence of a vector of real numbers $\mathbf{x}_0=(x_{0,1},x_{0,2},\ldots,x_{0,m})$ such that
\begin{equation}\label{un}
b_j=a_j e^{<\mathbf{r}_j,\mathbf{x}_0>i}=a_je^{i\sum_{k=1}^{m}r_{j,k}x_{0,k}},\ j=1,2,\ldots,n.
\end{equation}
Now, as the numbers $
c_{k}=\frac{g_k}{2\pi q},\text{ }k=1,2,\ldots,m,
$ are rationally independent,
we next apply Kronecker's theorem
\cite[p.382]{Hardy} with the following choice:
$c_k$, $\varepsilon_1=\frac{\varepsilon}{a\cdot m\cdot n\cdot r\cdot E}>0$ and
$
d_{k}=\frac{x_{0,k}}{2\pi q},\text{
}k=1,2,\ldots,m.
$
In this manner we assure the existence of a real number $\tau>d>0$ and integer numbers
$e_{1},e_{2},\ldots,e_{m}$ such that
\[
\left\vert \tau c_k-e_k-d_k\right\vert=\left|\frac{\tau g_k}{2\pi q}-e_k-\frac{x_{0,k}}{2\pi q}\right| <\varepsilon_1,
\]
that is
\begin{equation}\label{un2}
\tau g_k=2\pi qe_k+x_{0,k}+\eta_k, \mbox{with  }|\eta_k|<\varepsilon_1.
\end{equation}
Therefore, from (\ref{u}) and (\ref{un}), with $t\in \mathbb{R}$, we have
$$\sum_{j=1}^n|a_je^{i\lambda_j\tau}-b_j|=\sum_{j=1}^n\left|a_je^{i\lambda_j\tau}- a_je^{i\sum_{k=1}^{m}r_{j,k}x_{0,k}}\right|\leq$$
$$\sum_{j=1}^n|a_j|\left|e^{i\tau\lambda_j}-e^{i\sum_{k=1}^{m}r_{j,k}x_{0,k}}\right|\leq a\sum_{j=1}^n \left|e^{i\tau\lambda_j}-e^{i\sum_{k=1}^{m}r_{j,k}x_{0,k}}\right|=$$
$$a \sum_{j=1}^n\left|e^{i\tau\sum_{k=1}^{m} r_{j,k}g_k}-e^{i\sum_{k=1}^{m}r_{j,k}x_{0,k}}\right|,$$
which, from (\ref{un2}), is equal to
$$a \sum_{j=1}^n\left|e^{i\sum_{k=1}^{m} (r_{j,k}2\pi qe_k+r_{j,k}x_{0,k}+r_{j,k}\eta_k)}-e^{i\sum_{k=1}^{m}r_{j,k}x_{0,k}}\right|=$$
$$a \sum_{j=1}^n\left|e^{i\sum_{k=1}^{m}r_{j,k}\eta_k}-1\right|\leq a \sum_{j=1}^n\left|\sum_{k=1}^{m}r_{j,k}\eta_k\right|\leq$$
$$anr \sum_{k=1}^{m}\left|\eta_k\right|<anr \sum_{k=1}^{m}\frac{\varepsilon}{a\cdot m\cdot n\cdot r}=\varepsilon.$$
\end{proof}

As an immediate consequence of Theorem \ref{prop3}, we obtain the following corollary (compare with \cite[Corollary 3]{SV}). %--comprobar número

\begin{corollary}\label{corol3}
Given $\Lambda=\{\lambda_1,\lambda_2,\ldots,\lambda_n\}$ a finite set of exponents, let
$f_1(t)=\sum_{j=1}^na_je^{i\lambda_j t}$ and $f_2(t)=\sum_{j=1}^nb_je^{i\lambda_j t}$
be two equivalent functions in the class $\mathcal{P}_{\mathbb{R},\Lambda}$. Fixed $\varepsilon>0$, there exists a relatively dense set of real numbers $\tau$ such that
$$|f_1(t+\tau)-f_2(t)|<\varepsilon\ \ \forall t\in\mathbb{R}.$$
\end{corollary}
\begin{proof}Fixed $\tau>0$, note that for any $t\in\mathbb{R}$ it is accomplished that
$$|f_1(t+\tau)-f_2(t)|\leq \sum_{j=1}^n|a_je^{i\lambda_j(t+\tau)}-b_je^{i\lambda_jt}|=\sum_{j=1}^n|a_je^{i\lambda_j\tau}-b_j|.$$
Thus, by Theorem \ref{prop3} and given $d>0$, there exists $\tau_1>d$ such that
\begin{equation}\label{jhg}
|f_1(t+\tau_1)-f_2(t)|< \varepsilon/2\ \forall t\in\mathbb{R}.
\end{equation}
Moreover, since $f_1(t)$ is almost periodic, there exists a real number $l=l(\varepsilon)$% %<--cuidado, ¿especificar qué tipo de casi periodicidad?
\ such that every interval of length $l$
contains at
least one translation number $\tau$, associated with $\varepsilon $, satisfying
\begin{equation}\label{bgt}
\left\vert f_1(t+\tau)-f_1(t)\right\vert \leq \varepsilon/2\ \mbox{for all }t\in\mathbb{R}.
 \end{equation}
 %with $\sigma_0\leq\sigma\leq\sigma_1$ and $t\in \mathbb{R}$.
Consequently, from (\ref{jhg}) and (\ref{bgt}) we deduce the existence of a relatively dense set of real numbers $\tau$ such that any $t\in\mathbb{R}$ satisfies
$$|f_1(t+\tau+\tau_1)-f_2(t)|\leq |f_1(t+\tau_1+\tau)-f_1(t+\tau_1)|+|f_1(t+\tau_1)-f_2(t)|<\varepsilon.$$
This proves the result.
\end{proof}

It was proved in \cite[Proposition 2]{SV} that, with respect to the topology of uniform convergence, the equivalence classes in $\mathcal{P}_{\mathbb{R},\Lambda}/\shortstack{$_{{\fontsize{6}{7}\selectfont *}}$\\$\sim$}$ are sequentially compact. We can analogously prove that this property is also true with respect to the topology of $B^2(\mathbb{R},\mathbb{C})$ (see the proof in \cite[Proposition 2]{SV}).

\begin{proposition}\label{proppp}
Let $\Lambda$ be a finite set of exponents and $\mathcal{G}$ an equivalence class in  $\mathcal{P}_{\mathbb{R},\Lambda}/\shortstack{$_{{\fontsize{6}{7}\selectfont *}}$\\$\sim$}$. Thus $\mathcal{G}$ is sequentially compact.
\end{proposition}

\section{Besicovitch almost periodic functions in terms of an equivalence relation}

For our purposes, we next focus our attention on the Besicovitch space $B(\mathbb{R},\mathbb{C})$, whose functions are obtained by the completion of the trigonometric polynomials with respect to the seminorm
$\limsup_{l\to\infty}\left(\frac{ 1}{ 2l}\int_{-l}^{l}|f(t)|dt\right)$ (see for example \cite[Section 3.4]{Corduneanu}). In particular, the space of functions $B(\mathbb{R},\mathbb{C})$ contains those of the space of the almost periodic functions $AP(\mathbb{R},\mathbb{C})$ and those functions of $B^2(\mathbb{R},\mathbb{C})$. We recall that every function in $B(\mathbb{R},\mathbb{C})$ is associated with a real exponential sum with real frequencies of the form $\sum_{j\geq 1}a_je^{i\lambda_jt}$, which is called its Fourier series. %In general, when we write that a function $f$ is in these spaces we do not have in mind the function $f$ itself, it does represent a whole class of equivalent functions according to the relation $f_1\simeq f_2$ if and only if $\limsup_{l\to\infty}\left(\frac{ 1}{ 2l}\int_{-l}^{l}|f(t)-g(t)|dt\right)=0.$

%\begin{definition}\label{DF}
%Let $\Lambda=\{\lambda_1,\lambda_2,\ldots,\lambda_j,\ldots\}$ be an arbitrary countable set of distinct real numbers. We will say that a function $f:U\subset\mathbb{C}\to\mathbb{C}$ (resp. $f:\mathbb{R}\to\mathbb{C}$) is in the class $\mathcal{D}_{\Lambda}$ (resp. $\mathcal{F}_{\Lambda}$) if it is an almost periodic function in $B(U,\mathbb{C})$ (resp. in $B(\mathbb{R},\mathbb{C})$) whose associated Dirichlet series (resp. Fourier series) is of the form
% \begin{equation}\label{eqqo}
%\sum_{j\geq 1}a_je^{\lambda_js},\ a_j\in\mathbb{C}\setminus\{0\},\ \lambda_j\in\Lambda,
%\end{equation}
% \begin{equation}\label{eqq00o}
%\mbox{(resp. }\sum_{j\geq 1}a_je^{i\lambda_jt},\ a_j\in\mathbb{C}\setminus\{0\},\ \lambda_j\in\Lambda.\mbox{)}
%\end{equation}
%where $U$ is a strip of the type $\{s\in\mathbb{C}: \alpha<\operatorname{Re}s<\beta\}$, with $-\infty\leq\alpha<\beta\leq\infty$.
%\end{definition}

\begin{definition}\label{DF}
Let $\Lambda=\{\lambda_1,\lambda_2,\ldots,\lambda_j,\ldots\}$ be an arbitrary countable set of distinct real numbers. We will say that a function $f:\mathbb{R}\to\mathbb{C}$ is in the class $\mathcal{F}_{B^2,\Lambda}$ if it is an almost periodic function in $B^2(\mathbb{R},\mathbb{C})$ whose associated Fourier series is of the form
 \begin{equation}\label{eqq00o}
\sum_{j\geq 1}a_je^{i\lambda_jt},\ a_j\in\mathbb{C},\ \lambda_j\in\Lambda.
\end{equation}
\end{definition}

%It is convenient to recall that any almost periodic function in $AP(U,\mathbb{C})$ (resp. in $AP(\mathbb{R},\mathbb{C})$) is determined by its Dirichlet series (resp. Fourier series), which is of type (\ref{eqqo}) (resp. of type (\ref{eqq00o})). In fact, even in the case that the sequence of the partial sums of its Dirichlet series (resp. Fourier series) does not converge uniformly, there exists a sequence of finite exponential sums, called Bochner-Fej\'{e}r polynomials, of the type $P_k(s)=\sum_{j\geq 1}p_{j,k}a_je^{\lambda_js}$ (resp. $P_k(t)=\sum_{j\geq 1}p_{j,k}a_je^{i\lambda_jt}$) where for each $k$ only a finite number of the factors $p_{j,k}$ differ from zero, which converges uniformly to $f$ in every reduced strip in $U$ (resp. in $\mathbb{R}$) and converges formally to the Dirichlet series on $U$ (resp. to the Fourier series on $\mathbb{R}$) \cite[Polynomial approximation theorem, pgs. 50,148]{Besi}.

It is worth noting that, in general, when we write that a function $f$ is in $B(\mathbb{R},\mathbb{C})$ we do not have in mind the function $f$ itself, it does represent a whole class of equivalent functions according to the relation $f_1\simeq f_2$ if and only if $$\limsup_{l\to\infty}\left(\frac{ 1}{ 2l}\int_{-l}^{l}|f(t)-g(t)|^2dt\right)=0.$$

%The generalized Bohr's equivalence relation, for exponential sums in $\mathcal{S}_{\Lambda}$, of Definition \ref{DefEquiv00}, can be extended to their associated functions in B(R;C) and in particular to the
%classes F B 2

In terms of Definition \ref{DefEquiv00}, we can define an equivalence relation on the functions in $B(\mathbb{R},\mathbb{C})$, in particular on the classes $\mathcal{F}_{B^2,\Lambda}$. %-%In fact, we could say that two functions that are identifiable by its Dirichlet or Fourier series, with the same set of exponents, are equivalent when these series verify Definition \ref{DefEquiv}.
More specifically, we establish the following definition. %(see also \cite[Definition 5]{SV}). %--completar número HECHO

\begin{definition}\label{DefEquiv2}
Given $\Lambda=\{\lambda_1,\lambda_2,\ldots,\lambda_j,\ldots\}$ a set of exponents, let $f_1$ and $f_2$ denote two  equivalence classes of $B(\mathbb{R},\mathbb{C})/\simeq$  whose associated Fourier series are given by
 \begin{equation*}\label{eqq00}
\sum_{j\geq 1}a_je^{i\lambda_jt}\ \mbox{and}\ \sum_{j\geq 1}b_je^{i\lambda_jt},\ a_j,b_j\in\mathbb{C},\ \lambda_j\in\Lambda.
\end{equation*}
We will say that $f_1$ is equivalent to $f_2$ if for each integer value $n\geq 1$
there exists a $\mathbb{Q}$-linear map $\psi_n:V_n\to\mathbb{R}$, where $V_n$ is the $\mathbb{Q}$-vector space generated by $\{\lambda_1,\lambda_2,\ldots,\lambda_n\}$, such that
$$b_j=a_je^{i\psi_n(\lambda_j)},\ j=1,\ldots,n.$$
In that case, we will write $f_1\shortstack{$_{{\fontsize{6}{7}\selectfont *}}$\\$\sim$} f_2$.
\end{definition}

%The reader can observe that this equivalence relation is based on that of \cite[p.173]{Apostol} which was defined for general Dirichlet series. Bohr used it in that case in order to get so-called Bohr's equivalence theorem.

%It is plain that if two functions in $B(\mathbb{R},\mathbb{C})$ are identifiable in terms of their Fourier series then they are also equivalent in the sense of the definition above. That is, an equivalence class in $B(\mathbb{R},\mathbb{C})/\sim$ contains all the functions which are identifiable by their Fourier series.

%We next demonstrate the following important lemma which will let us consider equivalence classes, respect to the previous definition of $\sim$ (generalized Bohr's equivalence relation), in the space $B^2(\mathbb{R},\mathbb{C})$. It is clearly a consequence of Riesz-Fischer theorem \cite[p. 110]{Besi}.

The next important lemma allows us to prove that if a function $f_2$ is equivalent (in the sense of Definition \ref{DefEquiv2}) to a function $f_1$ belonging to the space $B^2(\mathbb{R},\mathbb{C})$, then $f_2$ also belongs to $B^2(\mathbb{R},\mathbb{C})$. This is clearly a consequence of Riesz-Fischer theorem \cite[p. 110]{Besi}.

\begin{lemma}\label{defnueva}
Let $f_1(t)\in B^2(\mathbb{R},\mathbb{C})$ be an almost periodic function whose Fourier series is given by $\sum_{j\geq 1}a_je^{i\lambda_jt},\ a_j\in\mathbb{C}$, where $\{\lambda_1,\ldots,\lambda_j,\ldots\}$ is a set of distinct exponents. Consider $b_j\in\mathbb{C}$ such that $\sum_{j\geq 1}b_je^{i\lambda_jt}$ and $\sum_{j\geq 1}a_je^{i\lambda_jt}$ are equivalent. Then $\sum_{j\geq 1}b_je^{i\lambda_jt}$ is the Fourier series associated with an almost periodic function $f_2(t)\in B^2(\mathbb{R},\mathbb{C})$ so that $f_1\shortstack{$_{{\fontsize{6}{7}\selectfont *}}$\\$\sim$} f_2$.
\end{lemma}
\begin{proof}
Take $\Lambda=\{\lambda_1,\ldots,\lambda_j,\ldots\}$. By the hypothesis, $f_1\in\mathcal{F}_{B^2,\Lambda}\subset B^2(\mathbb{R},\mathbb{C})$ is determined by the series
$\sum_{j\geq 1}a_je^{i\lambda_jt},\ a_j\in\mathbb{C},\ \lambda_j\in\Lambda.$ Moreover,
since $\sum_{j\geq 1}a_je^{i\lambda_jt}\shortstack{$_{{\fontsize{6}{7}\selectfont *}}$\\$\sim$} \sum_{j\geq 1}b_je^{i\lambda_jt}$, we deduce from Corollary \ref{use}
%for each $n\in\mathbb{N}$ there exists $\mathbf{x}_n\in\mathbb{R}^{\sharp\Lambda}$ such that $b_j=a_j e^{<\mathbf{r}_j,\mathbf{x}_n>i}$ for $j= 1,\ldots,n$, where $\mathbf{r}_j$ is given by (\ref{errej}). This implies
that $|b_j|=|a_j|$ for $j\geq 1$ and hence
$$\sum_{j\geq 1}|b_j|^2=\sum_{j\geq 1}|a_j|^2<\infty.$$
By Riesz-Fischer theorem \cite[p. 110]{Besi}, there exists a function $f_2\in B^2(\mathbb{R},\mathbb{C})$ such that the values $b_n$ are the Fourier coefficients of $f_2$.
\end{proof}

As it was said before, it is worth noting that a Fourier series $\sum_{n\geq 1}a_ne^{i\lambda_nt}$, such that $\sum_{n\geq 1}|a_n|^2<\infty$, represents an equivalence class (according to the relation $\simeq$) of functions in $B^2(\mathbb{R},\mathbb{C})$ (not a single function).
In fact, as we pointed out in introduction, since two almost periodic functions in the Besicovitch sense are connected in $B^2(\mathbb{R},\mathbb{C})$ when they have the same Fourier series (\cite[p. 148]{Besi} or \cite[Section 4.2]{Corduneanu}), we immediately deduce from the results above the following corollary.

%that if two functions in $B(\mathbb{R},\mathbb{C})$ are equivalent and we know that one of them is in $B^2(\mathbb{R},\mathbb{C})$, then both functions are in $B^2(\mathbb{R},\mathbb{C})$.

\begin{corollary}
Let $f_1(t)$ and $f_2(t)$ be two equivalent functions in $B(\mathbb{R},\mathbb{C})$. If $f_1(t)\in B^2(\mathbb{R},\mathbb{C})$, then $f_2(t)\in B^2(\mathbb{R},\mathbb{C})$.
\end{corollary}

%\subsection*{On the space $AP(\mathbb{R},\mathbb{C})$:}

% From this section, the set of functions of the Besicovitch space $B(\mathbb{R},\mathbb{C})\supset B^2(\mathbb{R},\mathbb{C})$ will be taken as the set of reference in the sense that each function in $B(\mathbb{R},\mathbb{C})$ is associated with a Fourier series \cite[Section 4.2]{Corduneanu}.
 The following result is concerned with the concept of convergence in $B^2(\mathbb{R},\mathbb{C})$ which is certainly weaker than the uniform convergence.
Under this topology, we next show that the equivalence classes of $\mathcal{F}_{B^2,\Lambda}/\shortstack{$_{{\fontsize{6}{7}\selectfont *}}$\\$\sim$}$ are closed. In fact, more specifically, they are sequentially compact. %The proof is analogous to that of \cite[Proposition 3]{SV}. %--comprobar núm
%In this respect, it is also worth noting that, by Lemma \ref{defnueva}, if $f\in \mathcal{F}_{B^2,\Lambda}$, then any function of its equivalence class is also included in $\mathcal{F}_{B^2,\Lambda}$.

 %In fact, more generally, they are sequentially compact.

\begin{proposition}\label{prop}
Let $\Lambda$ be a set of exponents and $\mathcal{G}$ an equivalence class in $\mathcal{F}_{B^2,\Lambda}/\shortstack{$_{{\fontsize{6}{7}\selectfont *}}$\\$\sim$}$. Thus $\mathcal{G}$ is sequentially compact.
\end{proposition}
\begin{proof}
Let $\{f_l\}_{l\geq 1}$ be a sequence in an equivalence class $\mathcal{G}$ in $\mathcal{F}_{B^2,\Lambda}/\shortstack{$_{{\fontsize{6}{7}\selectfont *}}$\\$\sim$}$.
For each $l=1,2,\ldots$, suppose that the Fourier series which is associated with  $f_l(t)$ is given by
$$\sum_{j\geq 1}a_{l,j}e^{i\lambda_jt}\ \mbox{with }a_{l,j}\in\mathbb{C},\ \lambda_j\in\Lambda.$$
Fixed a basis $G_{\Lambda}=\{g_1,g_2,\ldots,g_k,\ldots\}$ for $\Lambda$, let $\mathbf{r}_j=(r_{j,1},r_{j,2},\ldots)$ be the vector satisfying $<\mathbf{r}_j,\mathbf{g}>=\lambda_j$ for each $j\geq 1$, where $\mathbf{g}=(g_1,g_2,\ldots,g_k,\ldots)$. %is the vector of the basis for $\Lambda$.
Since $f_1\shortstack{$_{{\fontsize{6}{7}\selectfont *}}$\\$\sim$} f_l$ for each $l=1,2,\ldots$, we deduce from Proposition \ref{DefEquiv} that for each integer value $n\geq 1$ there exists  $\mathbf{x}_{l,n}=(x_{l,n,1},x_{l,n,2},\ldots)\in\mathbb{R}^{\sharp G_\Lambda}$ such that
\begin{equation}\label{seqaddpipipi1}
a_{l,j}=a_{1,j}e^{i<\mathbf{r}_j,\mathbf{x}_{l,n}>},\ j=1,2\ldots,n.
\end{equation}
Given $l\geq 1$, let $P_{l,k}(t)=\sum_{j\geq 1}p_{j,k}a_{l,j}e^{i\lambda_jt}$, $k=1,2,\ldots$, be the Bochner-Fej\'{e}r polynomials which converge to $f_l$ with respect to the topology of $B^2(\mathbb{R},\mathbb{C})$ (and converge formally to its Fourier series on $\mathbb{R}$) \cite[p. 105, Theorem II]{Besi}. It is worth noting that for each $k$ only a finite number of the factors $p_{j,k}$ differ from zero, and these factors $p_{j,k}$ do not depend on $l$ \cite[p. 48]{Besi}. Thus, by taking into account (\ref{seqaddpipipi1}), it is clear that $\{P_{l,1}(t)\}_{l\geq 1}$ is a sequence of equivalent trigonometric polynomials and, by Proposition \ref{proppp},
there exists a subsequence $\{P_{l_{m,1},1}(t)\}_{m\geq 1}\subset \{P_{l,1}(t)\}_{l\geq 1}$ convergent to a certain $P_1(t)=\sum_{j\geq 1}p_{j,1}a_{j}e^{i\lambda_jt}\in \mathcal{P}_{\mathbb{R},\Lambda_1}$, where $\Lambda_1=\{\lambda_j\in\Lambda:p_{j,1}\neq 0\}$, which is in the same equivalence class as $P_{1,1}(t)$.
Furthermore, by Proposition \ref{DefEquiv}, this means that there exists  $\mathbf{x}_0^{(1)}=(x_{0,1}^{(1)},x_{0,2}^{(1)},\ldots)\in\mathbb{R}^{m_1}$ such that
\begin{equation*}
p_{j,1}a_{j}=p_{j,1}a_{1,j}e^{i<\mathbf{r}_j,\mathbf{x}_0^{(1)}>},\ j=1,2\ldots,\mbox{ with }\lambda_j\in\Lambda_1,
\end{equation*}
where $m_1$ is the number of elements of any basis for $\Lambda_1$.
Equivalently
\begin{equation*}%\label{seqaddpipipipi}
a_{j}=a_{1,j}e^{i<\mathbf{r}_j,\mathbf{x}_0^{(1)}>},\ j=1,2\ldots,\mbox{ with }\lambda_j\in\Lambda_1.
\end{equation*}
Analogously, from the sequence $\{P_{l_{m,1},2}(t)\}_{m\geq 1}$, we can draw a subsequence $\{P_{l_{m,2},2}(t)\}_{m\geq 1}\subset \{P_{l_{m,1},2}(t)\}_{m\geq 1}$ convergent to a certain $$P_2(t)=\sum_{j\geq 1}p_{j,2}a_{j}e^{i\lambda_jt}\in \mathcal{P}_{\mathbb{R},\Lambda_2},$$ where $\Lambda_2=\{\lambda_j\in\Lambda:p_{j,2}\neq 0\}\cup\Lambda_1$, which is in the same equivalence class as $P_{1,2}(t)$. This implies that there exists $\mathbf{x}_0^{(2)}=(x_{0,1}^{(2)},x_{0,2}^{(2)},\ldots)\in\mathbb{R}^{m_2}$ such that
\begin{equation*}%\label{seqaddpipipipi2}
a_{j}=a_{1,j}e^{i<\mathbf{r}_j,\mathbf{x}_0^{(2)}>},\ j=1,2\ldots,\mbox{ with }\lambda_j\in\Lambda_2,
\end{equation*}
where $m_2$ is the number of elements of any basis for $\Lambda_2$.
%In fact, the candidates of the first $m_1$ components of the vectors $\mathbf{x}_0^{(2)}$ are subsets of the vectors $\mathbf{x}_0^{(1)}$ above (see remark \ref{add}).
In general, for each $k=2,3,\ldots$, we can extract a subsequence $\{P_{l_{m,k},k}(t)\}_{m\geq 1}\subset \{P_{l_{m,k-1},k}(t)\}_{m\geq 1}$ convergent to a certain $$P_k(t)=\sum_{j\geq 1}p_{j,k}a_{j}e^{i\lambda_jt}\in \mathcal{P}_{\mathbb{R},\Lambda_k},$$ where $\Lambda_k=\{\lambda_j\in\Lambda:p_{j,k}\neq 0\}\cup\Lambda_{k-1}$, which is in the same equivalence class as $P_{1,k}(t)$ and hence  there exists  $\mathbf{x}_0^{(k)}=(x_{0,1}^{(k)},x_{0,2}^{(k)},\ldots)\in\mathbb{R}^{m_k}$ ($m_k$ is the number of elements of any basis for $\Lambda_k$) such that
\begin{equation}\label{seqaddpipipipi2}
a_{j}=a_{1,j}e^{i<\mathbf{r}_j,\mathbf{x}_0^{(k)}>},\ j=1,2\ldots,\mbox{ with }\lambda_j\in\Lambda_k.
\end{equation}
%Again, the candidates of the first $m_{k-1}$ components of the vectors $\mathbf{x}_0^{(k)}$ are subsets of the vectors $\mathbf{x}_0^{(k-1)}$ (see remark \ref{add}).
So we get by induction a sequence $\{P_k(t)\}_{k\geq 1}$ of trigonometric polynomials which converges formally to the series
\begin{equation}\label{bo}
\sum_{j\geq 1}a_{j}e^{i\lambda_jt},\ \lambda_j\in\Lambda,
\end{equation}
and, since (\ref{seqaddpipipipi2}) is satisfied for any $k=1,2,\ldots$, we can  construct, for each integer value $n\geq 1$, a vector $\mathbf{x}_{0,n}\in\mathbb{R}^{\sharp G_{\Lambda}}$ such that
\begin{equation*}%\label{seqaddpipipipi22}
a_{j}=a_{1,j}e^{i<\mathbf{r}_j,\mathbf{x}_{0,n}>},\ j=1,2\ldots,n\mbox{ with }\lambda_j\in\Lambda.
\end{equation*}%Por reducción al absurdo sale
Hence the series (\ref{bo}) is equivalent to $\sum_{j\geq 1}a_{1,j}e^{i\lambda_jt}$ and, by Lemma \ref{defnueva}, it is the Fourier series associated with an almost periodic function $h(t)\in B^2(\mathbb{R},\mathbb{C})$ such that $h\shortstack{$_{{\fontsize{6}{7}\selectfont *}}$\\$\sim$} f_1$. Consequently, $\{P_k(t)\}_{k\geq 1}$ converges
with respect to the topology of $B^2(\mathbb{R},\mathbb{C})$ to $h(t)\in\mathcal{G}$ and we can extract a subsequence of $\{f_l(t)\}_{l\geq1}$ which also converges in $B^2(\mathbb{R},\mathbb{C})$ to $h(t)$. %the result holds.
\end{proof}

%In particular, the set of the functions which are identifiable in $B^2(\mathbb{R},\mathbb{C})$ in terms of their Fourier series is also closed with respect to the norm in this space. %$B^2(\mathbb{R},\mathbb{C})$.
%That is, the limit of a convergent sequence of functions in $B^2(\mathbb{R},\mathbb{C})$ which have the same Fourier series is a function in $B^2(\mathbb{R},\mathbb{C})$ which has the same Fourier series.

As a consequence of Proposition \ref{prop}, in the topology of
$B^2(\mathbb{R},\mathbb{C})$, we next show that the family of translates of a function $f\in \mathcal{F}_{B^2,\Lambda}$ is closed on its equivalence class of $\mathcal{F}_{B^2,\Lambda}/\shortstack{$_{{\fontsize{6}{7}\selectfont *}}$\\$\sim$}$.
\begin{corollary}\label{cor5}
Let $\Lambda$ be a set of exponents and $f\in \mathcal{F}_{B^2,\Lambda}$. Thus the limit points of the set of functions $\mathcal{T}_f=\{f_{\tau}(t):=f(t+\tau):\tau\in\mathbb{R}\}$ are functions which are equivalent to $f$.
\end{corollary}
\begin{proof}
Since it is plain that the functions included in
$\mathcal{T}_f=\{f_{\tau}(p):=f(t+\tau):\tau\in\mathbb{R}\}$ are in the same equivalence class of $f$ (see in \cite[Section 4.2]{Corduneanu} the Fourier series of the translates of a function in the Besicovitch spaces), %--comprobar section ya sabes HECHO
the result follows easily from Proposition \ref{prop}.
\end{proof}

Now Corollary \ref{cor5} can be improved with the following result.
Indeed, we next prove that, fixed a function $f\in \mathcal{F}_{B^2,\Lambda}$, the limit points of the set of the translates $\mathcal{T}_f=\{f(t+\tau):\tau\in\mathbb{R}\}$ of $f$ are precisely the almost periodic functions which are equivalent to $f$.

\begin{theorem}\label{mth0}
Let $\Lambda$ be a set of exponents, $\mathcal{G}$ an equivalence class in $\mathcal{F}_{B^2,\Lambda}/\shortstack{$_{{\fontsize{6}{7}\selectfont *}}$\\$\sim$}$ and $f\in \mathcal{G}$. Thus the set of functions $\mathcal{T}_f=\{f_{\tau}(t):=f(t+\tau):\tau\in\mathbb{R}\}$ is dense in $\mathcal{G}$.
\end{theorem}
\begin{proof}
Let $f(t)$ be a function in the class $\mathcal{F}_{B^2,\Lambda}$.
We know by Corollary \ref{cor5} that the limit points of the set of functions $\mathcal{T}_f=\{f_{\tau}(t):=f(t+\tau):\tau\in\mathbb{R}\}$ are functions in $B^2(\mathbb{R},\mathbb{C})$ which are equivalent to $f$. We next demonstrate that
any function $h(t)$ which is equivalent to
$f(t)$ is also a limit point of $\mathcal{T}_f$.
If $\sharp\Lambda<\infty$, given $\varepsilon_n=\frac{1}{n}$, $n\in\mathbb{N}$, Corollary \ref{corol3} assures the existence of an increasing sequence $\{\tau_n\}_{n\geq 1}$ of positive real numbers such that any $n\in\mathbb{N}$ verifies
$$|f(t+\tau_n)-h(t)|^2<\varepsilon_n\ \forall t\in\mathbb{R}.$$
Hence $M(|f_{\tau_n}(t)-h(t)|^2)\to 0$ as $n$ goes to $\infty$ (see (\ref{Meanvalo}) for the definition of the mean value $M(f)$), and the result holds for the case $\sharp\Lambda<\infty$.
Consider $\sharp\Lambda=\infty$ and let $\sum_{j\geq 1}a_je^{i\lambda_j t}$ and $\sum_{j\geq 1}b_je^{i\lambda_j t}$ be the Fourier series of $f\in \mathcal{F}_{B^2,\Lambda}$ and $h\shortstack{$_{{\fontsize{6}{7}\selectfont *}}$\\$\sim$} f$, respectively. Take $\varepsilon_1=\sum_{j>1}|a_j|^2>0$, then  Theorem \ref{prop3} assures the existence of $\tau_1>0$ such that
$$\left|a_1e^{i\lambda_1(t+\tau_1)}-b_1e^{i\lambda_1t}\right|<\sqrt{\varepsilon_1}\ \forall t\in\mathbb{R},$$
which implies
\begin{equation}\label{aprop1}
\left|a_1e^{i\lambda_1\tau_1}-b_1\right|^2<\varepsilon_1.
\end{equation}
Thus, from (\ref{aprop1}) and $|a_j|=|b_j|$ for any $j\geq 1$ (Corollary \ref{use}), we have that
$$\sum_{j\geq 1}|a_je^{i\lambda_1\tau_1}-b_j|^2<\varepsilon_1+\sum_{j>1 }|a_je^{i\lambda_j\tau_1}-b_j|^2\leq \varepsilon_1+\sum_{j>1}(|a_j|+|b_j|)^2=$$$$\varepsilon_1+4\sum_{j>1}|a_j|^2=5\varepsilon_1.$$
Consequently,
$$M(|f_{\tau_1}(t)-h(t)|^2)<5\varepsilon_1.$$
Similarly, take $\varepsilon_2=\sum_{j>2}|a_j|^2>0$, then  Theorem \ref{prop3} assures the existence of $\tau_2>\tau_1$ such that
$$\sum_{j=1}^2\left|a_je^{i\lambda_j(t+\tau_2)}-b_je^{i\lambda_jt}\right|<\sqrt{\varepsilon_2},
$$
which implies
\begin{equation}\label{aprop2}
\left(\sum_{j=1}^2\left|a_je^{i\lambda_j\tau_2}-b_j\right|\right)^2<\varepsilon_2.
\end{equation}
Therefore, from (\ref{aprop2}) and $|a_j|=|b_j|$ for any $j\geq 1$, we have
$$\sum_{j\geq 1}|a_je^{i\lambda_1\tau_2}-b_j|^2=|a_1e^{i\lambda_1\tau_2}-b_1|^2+|a_2e^{i\lambda_1\tau_2}-b_2|^2+\sum_{j>2}|a_je^{i\lambda_j\tau_2}-b_j|^2\leq$$
$$(|a_1e^{i\lambda_1\tau_2}-b_1|+|a_2e^{i\lambda_1\tau_2}-b_2|)^2+\sum_{j>2}|a_je^{i\lambda_j\tau_2}-b_j|^2\leq$$
$$\leq \varepsilon_2+\sum_{j>2}(|a_j|+|b_j|)^2=\varepsilon_2+4\sum_{j>2}|a_j|^2=5\varepsilon_2.$$
Consequently,
$$M(|f_{\tau_2}(t)-h(t)|^2)<5\varepsilon_2.$$
In general, by repeating this process, we can construct an increasing sequence $\{\tau_n\}_{n\geq 1}$ such that each $\tau_n$ satisfies that
$$\sum_{j=1}^n\left|a_je^{i\lambda_j(t+\tau_n)}-b_je^{i\lambda_jt}\right|<\sqrt{\varepsilon_n},
$$
which implies
\begin{equation}\label{apropn}
\left(\sum_{j=1}^n\left|a_je^{i\lambda_j\tau_n}-b_j\right|\right)^2<\varepsilon_n.
\end{equation}
with $\varepsilon_n=\sum_{j>n}|a_j|^2$. Thus, from (\ref{apropn}) we have
$$M(|f_{\tau_n}(t)-h(t)|^2)=\sum_{j\geq 1}|a_je^{i\lambda_1\tau_n}-b_j|^2=$$
$$\sum_{j=1}^n|a_je^{i\lambda_j\tau_n}-b_j|^2+\sum_{j>n}|a_je^{i\lambda_j\tau_n}-b_j|^2\leq$$
$$\left(\sum_{j=1}^n|a_je^{i\lambda_j\tau_n}-b_j|\right)^2+\sum_{j>n}|a_je^{i\lambda_j\tau_n}-b_j|^2\leq$$
$$\leq \varepsilon_n+\sum_{j>n}(|a_j|+|b_j|)^2=\varepsilon_n+4\sum_{j>n}|a_j|^2=5\varepsilon_n.$$
Note that $\sum_{j\geq 1}|a_j|^2<\infty$, then
 $\sum_{j>n}|a_j|^2$ tends to $0$ when $n$ goes to $\infty$. Consequently, the sequence of functions
$\{f(t+\tau_n)\}_{n\geq 1}$ converges in $B^2(\mathbb{R},\mathbb{C})$ to $h(t)$, and the result holds.
\end{proof}

\begin{corollary}\label{cmth}
Let $f\in B^2(\mathbb{R},\mathbb{C})$ and $f_1\shortstack{$_{{\fontsize{6}{7}\selectfont *}}$\\$\sim$} f$.
There exists an increasing\-
unbounded sequence $\{\tau_n\}_{n\geq 1}$
of positive numbers such that the sequence of functions
$\{f(t+\tau_n)\}_{n\geq 1}$ converges in $B^2(\mathbb{R},\mathbb{C})$ to $f_1(t)$. In fact, given $\varepsilon>0$ there exists a satisfactorily uniform set of positive numbers $\tau$ such that
$$M(|f(t+\tau)-f_1(t)|^2)<\varepsilon.$$
%$$\limsup_{l\to\infty}\left((2l)^{-1}\int_{-l}^l|f(t+\tau_k)-f_1(t)|^2 dt\right)^{1/2}<\varepsilon.$$
\end{corollary}
\begin{proof}
Let $f\in B^2(\mathbb{R},\mathbb{C})$, then $f\in \mathcal{F}_{B^2,\Lambda}$ for some set $\Lambda$ of exponents. Let $\mathcal{G}$ be the equivalence class in $\mathcal{F}_{B^2,\Lambda}/\shortstack{$_{{\fontsize{6}{7}\selectfont *}}$\\$\sim$}$ so that $f\in\mathcal{G}$ and let $f_1\shortstack{$_{{\fontsize{6}{7}\selectfont *}}$\\$\sim$} f$. Thus, by Theorem \ref{mth0} (see also its proof), there exists an increasing unbounded sequence $\{\tau_n\}_{n\geq 1}$
of positive numbers such that the sequence of functions
$\{f(t+\tau_n)\}_{n\geq 1}$ converges in $B^2(\mathbb{R},\mathbb{C})$ to $f_1(t)$. Equivalently, given $\varepsilon>0$ there exists $n_0\in\mathbb{N}$ such that
$$M(|f(t+\delta_n)-f_1(t)|^2)<\varepsilon/2\ \forall n\geq n_0.$$
%$$|f(t+\delta_n)-f_1(t)|<\varepsilon/2\ \forall n\geq n_0,\ \forall t\in\mathbb{R}.$$
%$$\limsup_{l\to\infty}\left((2l)^{-1}\int_{-l}^l|f(t+\delta_n)-f_1(t)|^2 dt\right)^{1/2}<\varepsilon/2\ \forall n\geq n_0.$$
Moreover, since $f(t)$ is almost periodic in the sense of Besicovitch, %<--citar esto???
there exist a set $S=\{\tau_k\}\subset \mathbb{R}$  and $l=l(\varepsilon)>0$ such that the ratio of the maximum number of elements of $S$ included in an interval $(a,a+l)$ to the minimum number is less than $2$ and satisfy
$$M(|f(t+\tau_k)-f(t)|^2)<\varepsilon/2.$$
%$$\limsup_{l\to\infty}\left((2l)^{-1}\int_{-l}^l|f(t+\tau_k)-f(t)|^2 dt\right)^{1/2}<\varepsilon/2.$$
Hence any $\tau_k$ satisfies
\begin{multline*}
M(|f(t+\delta_n+\tau_k)-f_1(t)|^2)\leq M(|f(t+\delta_n+\tau_k)-f(t+\delta_n)|^2)+\\+M(|f(t+\delta_n)-f_1(t)|^2)
<\varepsilon\ \ \forall n\geq n_0,
\end{multline*}
%\begin{multline*}
%\limsup_{l\to\infty}\left((2l)^{-1}\int_{-l}^l|f(t+\delta_n+\tau_k)-f_1(t)|^2 dt\right)^{1/2}\leq\\
%\limsup_{l\to\infty}\left((2l)^{-1}\int_{-l}^l|f(t+\delta_n+\tau_k)-f(t+\delta_n)|^2 dt\right)^{1/2}+\\+
%\limsup_{l\to\infty}\left((2l)^{-1}\int_{-l}^l|f(t+\delta_n)-f_1(t)|^2 dt\right)^{1/2}<\varepsilon\ \forall n\geq n_0,
%\end{multline*}
which proves the result.

\end{proof}

It is known that the almost periodic functions in the Besicovitch spaces $B^p(\mathbb{R},\mathbb{C})$, $1\leq p<\infty$, satisfy the Bochner-type property consisting of the relative compactness of the set $\{f(t+\tau)\}$, $\tau\in\mathbb{R}$, associated with an arbitrary function $f\in B^p(\mathbb{R},\mathbb{C})$ (see \cite[Theorem 5.34]{Andres} or \cite[Section 3.4]{Corduneanu}).
As an important consequence of Theorem \ref{mth0}, we next refine this property for the case of $B^2(\mathbb{R},\mathbb{C})$ in the sense that we show that the condition of
almost periodicity, in the sense of Besicovitch, of a function $f(t)$ implies that every sequence $\{f(t+\tau_n)\}$, $\tau_n\in\mathbb{R}$, of translates of $f$ has a subsequence that converges with the topology of $B^2(\mathbb{R},\mathbb{C})$ to a function which is equivalent to $f$.

\begin{corollary}
If $f\in B^2(\mathbb{R},\mathbb{C})$, then the compact closure of its set of translates coincides with its equivalence class.
\end{corollary}
\begin{proof}
First of all, we recall that any function $f\in B(\mathbb{R},\mathbb{C})$ has an associated Fourier series.
Let $f\in B^2(\mathbb{R},\mathbb{C})$, then $f\in \mathcal{F}_{B^2,\Lambda}$ for some set $\Lambda$ of exponents. Now, let $\mathcal{G}$ be the equivalence class in $\mathcal{F}_{B^2,\Lambda}/\shortstack{$_{{\fontsize{6}{7}\selectfont *}}$\\$\sim$}$ so that $f\in\mathcal{G}$.
%Take now a sequence
% $\{f(s+i\tau_n)\}$, $\tau_n\in\mathbb{R}$, of
%vertical translations of $f$, which are in $\mathcal{F}$ (see Lemma \ref{lem}), and apply Theorem \ref{prop}. Thus there exists a subsequence that converges uniformly on the compact subsets of $U$ to a function which is in $\mathcal{F}$. In fact,
By Theorem \ref{mth0}, all the limit points of the translates of $f$ are exponential sums which are included in $\mathcal{G}$ and, in fact, the compact closure of the set of the translates of $f$ coincides with $\mathcal{G}$.
\end{proof}

\begin{remark}
Given $\Lambda=\{\lambda_1,\lambda_2,\ldots,\lambda_j,\ldots\}$ a set of exponents, consider $A_1(p)$ and $A_2(p)$ two exponential sums in the class $\mathcal{S}_{\Lambda}$, say
$A_1(p)=\sum_{j\geq1}a_je^{\lambda_jp}$ and $A_2(p)=\sum_{j\geq1}b_je^{\lambda_jp}.$ Let $V$ be the $\mathbb{Q}$-vector space generated by $\Lambda$.
We will say that $A_1$ is $B$-equivalent to $A_2$ if there exists a $\mathbb{Q}$-linear map $\psi_n:V\to\mathbb{R}$ such that
$$b_j=a_je^{i\psi_n(\lambda_j)},\ j=1,2,\ldots.$$
It is easy to prove that, fixed a basis $G_{\Lambda}$ for $\Lambda$, $A_1$ is $B$-equivalent to $A_2$ %, relative to the basis $G_{\Lambda}$,
if and only if there exists $\mathbf{x}_0=(x_{0,1},x_{0,2},\ldots,x_{0,k},\ldots)\in \mathbb{R}^{\sharp G_{\Lambda}}$
such that $b_j=a_j e^{<\mathbf{r}_j,\mathbf{x}_0>i}$ for every $j\geq 1$, where the $\mathbf{r}_j$'s are the vectors of rational components verifying (\ref{errej}). %In that case, we will write $A_1\sim A_2$.

From this and Proposition \ref{DefEquiv}, it is worth noting that Definition \ref{DefEquiv00} and definition of $B$-equivalence are equivalent in the case that it is possible to obtain an integral basis for the set of exponents $\Lambda$. Consequently, all the results of this paper which can be formulated in terms of an integral basis are also valid under the $B$-equivalence. (in particular, those related to the finite exponential sums in Section \ref{section3}).
\end{remark}

\bibliographystyle{amsplain}

\end{document}